\documentclass[12pt,leqno]{amsart}\usepackage[latin9]{inputenc}
\usepackage{verbatim,amsthm,amstext,amssymb,color,amscd,bm}
\makeatletter\numberwithin{equation}{section}
\theoremstyle{plain}\newtheorem{theorem}{Theorem}[section]
\newtheorem{lemma}[theorem]{Lemma}\makeatother

\begin{document}
\title[A new universal real flow of the Hilbert-cubical type]
{A new universal real flow of the\\Hilbert-cubical type}

\author[L. Jin]{Lei Jin}
\address{Lei Jin: Department of Mathematics, University of Science and Technology of China,
Hefei, Anhui 230026, China}
\email{jinleim@mail.ustc.edu.cn}

\author[S. Tu]{Siming Tu $^*$}\let\thefootnote\relax\footnote{* Corresponding author.}
\address{Siming Tu: School of Mathematics (Zhuhai), Sun Yat-sen University,
Zhuhai, Guangdong 519082, China}
\email{tusiming3@mail.sysu.edu.cn}


\begin{abstract}
We provide a new universal real flow of the Hilbert-cubical type.
We prove that any real flow can be equivariantly embedded in the translation on $L(\mathbb{R})^\mathbb{N}$,
where $L(\mathbb{R})$ denotes the space of $1$-Lipschitz functions $f:\mathbb{R}\to[0,1]$.
Furthermore, all those functions in $L(\mathbb{R})^\mathbb{N}$
that are images of such embeddings can be chosen as $C^1$-functions.
\end{abstract}

\maketitle

\section{Introduction}
By a \textit{flow} we understand a pair $(X,(T_t)_{t\in\mathbb{R}})$
(sometimes written as $(X,\mathbb{R})$),
where $X$ is a compact metric space and the map
$$T:\,\mathbb{R}\times X\to X,\,\;\,(t,x)\mapsto T_tx$$
is continuous.
We say that $(X,(T_t)_{t\in\mathbb{R}})$ can be \textit{embedded} in $(Y,(S_t)_{t\in\mathbb{R}})$
if there exists an equivariant topological embedding $\phi:X\to Y$,
namely, a homeomorphism $\phi$ of $X$ into $Y$ such that
$\phi(T_tx)=S_t\phi(x)$ for all $x\in X$ and $t\in\mathbb{R}$.
A flow is called \textit{universal} if
all flows can be embedded in it.

One can similarly apply the terminology to $\mathbb{Z}$-actions.
Since every compact metric space can be (topologically) embedded
in $[0,1]^\mathbb{N}$ (see \cite{M}),
there is a rather natural universal $\mathbb{Z}$-action:
the translation on the Hilbert cube
$([0,1]^\mathbb{N})^\mathbb{Z}$,
denoted by $(([0,1]^\mathbb{N})^\mathbb{Z},\sigma)$,
where $([0,1]^\mathbb{N})^\mathbb{Z}$ is equipped with the product topology and
$\sigma$ acts on it continuously by the translation:
$$\sigma((x_k)_{k\in\mathbb{Z}})=(x_{k+1})_{k\in\mathbb{Z}},\,\;\,(x_k\in[0,1]^\mathbb{N}).$$
Note that the Hilbert cube
$([0,1]^\mathbb{N})^\mathbb{Z}$
is a compact metric space.

Let us return to the case of $\mathbb{R}$-actions.
It is natural to seek a universal flow as well.
Since the Hilbert cube $([0,1]^\mathbb{N})^\mathbb{Z}$ can be written as
$([0,1]^\mathbb{Z})^\mathbb{N}$ and $[0,1]^\mathbb{Z}$
can be viewed as the space of continuous functions
$f:\mathbb{Z}\to[0,1]$,
we may consider the candidate $C(\mathbb{R})^\mathbb{N}$,
where $C(\mathbb{R})$ is the space of continuous functions $f:\mathbb{R}\to[0,1]$
equipped with the topology of uniform convergence on compact subsets of $\mathbb{R}$,
given by the distance
\begin{equation}\label{distance}
\sum_{n=1}^\infty\frac{1}{2^n}\max_{|t|\le n} |f(t)-g(t)|,\,\;\,(f,g\in C(\mathbb{R})).
\end{equation}
We let the group $\mathbb{R}$ act on $C(\mathbb{R})$ continuously by the translation:
\begin{equation}\label{translation}
\mathbb{R}\times C(\mathbb{R})\to C(\mathbb{R}),\,\;\,(s,f(t))\mapsto f(t+s).
\end{equation}
In the same way as $\mathbb{Z}$-actions,
we can embed all flows in the translation on the product space $C(\mathbb{R})^\mathbb{N}$.
From this point of view,
the space $C(\mathbb{R})^\mathbb{N}$ seems to be a natural correspondence
to the discrete case of the Hilbert cube $([0,1]^\mathbb{N})^\mathbb{Z}$
with the same universal property.
However, there is a drawback:
The space $C(\mathbb{R})^\mathbb{N}$ is neither compact nor locally compact.
So it is actually not a ``flow'' in the definition.

Given this motivation,
we were interested in finding a universal flow
whose state space is a compact subspace of $C(\mathbb{R})^\mathbb{N}$.
This poses the following question:
\begin{enumerate}
\item[$\bullet$]
Is there an ``explicit'' compact invariant subset of $C(\mathbb{R})^\mathbb{N}$
that is universal?
\end{enumerate}
Here ``explicitness'' means that we may easily characterize
all elements in the space that we choose.
Answering this question affirmatively,
\cite{GJ} constructed a compact invariant subset
$\prod_{n\in\mathbb{N}}B_1V_{-n}^n\subset C(\mathbb{R})^\mathbb{N}$
that is universal under the action of translation,
where $B_1V_{-n}^n$ denotes the collection of all elements
in $C(\mathbb{R})$ whose Fourier transform
(considered as a tempered distribution) is supported in $[-n,n]$.
However, in contrast to the Hilbert cube $([0,1]^\mathbb{N})^\mathbb{Z}$,
it is not so simple, and it is not a countable self-product.

We expect such a universal flow, in addition, to be a self-product
(and then it will be ``closer'' to the Hilbert cube $([0,1]^\mathbb{N})^\mathbb{Z}$).
More precisely, our problem appears as follows:
\begin{enumerate}
\item[$\bullet$]
Is there an ``explicit'' compact invariant subset $F$ of $C(\mathbb{R})$
such that $F^\mathbb{N}$ is universal under the action of translation?
\end{enumerate}

The purpose of this paper is to solve this problem affirmatively.
We provide a new explicit universal flow that is of the Hilbert-cubical type
(i.e., a universal flow of the form $F^\mathbb{N}$
with a compact invariant subset $F\subset C(\mathbb{R})$).
As we mentioned previously,
$[0,1]^\mathbb{Z}$ can be regarded as
the space of continuous functions
$f:\mathbb{Z}\to[0,1]$;
and moreover,
an easy observation reveals that
all functions $f:\mathbb{Z}\to[0,1]$
are automatically $1$-Lipschitz,
meaning that
$|f(x)-f(y)|\le|x-y|$ for all $x,y\in\mathbb{Z}$.
Thus, for flows we may consider the space of
$1$-Lipschitz functions $f:\mathbb{R}\to[0,1]$.

Formally, let $L(\mathbb{R})$ be the space of all functions $f:\mathbb{R}\to[0,1]$
satisfying the $1$-Lipschitz condition:
$$|f(s)-f(t)|\le|s-t|\,\;\,\text{ for all }\,s,t\in\mathbb{R}.$$
It is clear that $L(\mathbb{R})$ is a subset of $C(\mathbb{R})$.
An important fact is that it is compact and invariant:
\begin{lemma}The space $L(\mathbb{R})$ is a compact metric space
with respect to the distance \eqref{distance},
and is invariant under the action of translation
\eqref{translation}.\end{lemma}\begin{proof}
The former statement follows from the Arzela--Ascoli theorem,
and the latter statement follows immediately
from the definition of the space.\end{proof}Therefore,
under the action of translation \eqref{translation},
the space $L(\mathbb{R})$ is a flow,
and so is the product space $L(\mathbb{R})^\mathbb{N}$
which corresponds to the Hilbert cube
$([0,1]^\mathbb{N})^\mathbb{Z}=([0,1]^\mathbb{Z})^\mathbb{N}$.
We denote by $(L(\mathbb{R})^\mathbb{N},\mathbb{R})$
the translation on $L(\mathbb{R})^\mathbb{N}$.
Our main result is that the flow $(L(\mathbb{R})^\mathbb{N},\mathbb{R})$
is universal.\begin{theorem}[Main theorem]\label{main}
Any flow can be embedded in $(L(\mathbb{R})^\mathbb{N},\mathbb{R})$.
\end{theorem}

\section{A constructive proof of the main theorem}
\begin{proof}[Proof of Theorem \ref{main}]
Let $(X,\mathbb{R})$ be a flow.
For convenience we denote by $(T_t)_{t\in\mathbb{R}}$
the translation on $C(\mathbb{R})$ or $C(\mathbb{R})^\mathbb{N}$.
We can treat $C(\mathbb{R})^\mathbb{N}$ as
the space of continuous functions
$f:\mathbb{R}\to[0,1]^\mathbb{N}$.

\medskip

We can assume that $X$ is a compact invariant subset of
$C(\mathbb{R})^\mathbb{N}$.
In fact, since $X$ is a compact metric space,
there is a topological embedding $\psi$
of $X$ in $[0,1]^\mathbb{N}$,
namely,
a homeomorphism $\psi:X\to\psi(X)\subset[0,1]^\mathbb{N}$.
We write $(X,\mathbb{R})=(X,(\varphi_t)_{t\in\mathbb{R}})$
and define a map
$\phi:X\to C(\mathbb{R})^\mathbb{N}$ as follows:
For every $x\in X$, we define
$\phi(x):\mathbb{R}\to[0,1]^\mathbb{N}$ by
$\phi(x)(t)=\psi(\varphi_t(x))$ for all $t\in\mathbb{R}$.
It is clear that
for every $x\in X$,
$\phi(x)$ is in $C(\mathbb{R})^\mathbb{N}$,
and that the map $\phi:X\to\phi(X)$ is one-to-one.
To see that $\phi$ is continuous,
we fix a point $x\in X$ and take a compact subset $A$ of $\mathbb{R}$ and
a sequence of points $x_n\in X$ tending to $x$ as $n\to+\infty$.
Since $X$ and $A\times X$ are compact,
the maps $\psi$ and $\varphi$ are uniformly continuous on $X$ and $A\times X$,
respectively.
Thus,
when $n$ is large enough,
$\varphi_t(x_n)$
is sufficiently close to
$\varphi_t(x)$
for all $t\in A$,
which implies that
$\phi(x_n)(t)=\psi(\varphi_t(x_n))$
is sufficiently close to
$\phi(x)(t)=\psi(\varphi_t(x))$
for all $t\in A$.
This shows that the sequence
$\phi(x_n)\in C(\mathbb{R})^\mathbb{N}$
uniformly tends to $\phi(x)\in C(\mathbb{R})^\mathbb{N}$
on $A$ as $n\to+\infty$.
Therefore the map
$\phi$ is continuous.
Since $X$ is compact and
$\phi:X\to\phi(X)$ is continuous and one-to-one,
we have that the map $\phi:X\to\phi(X)$ is a homeomorphism.
Meanwhile,
since for any $x\in X$ and $r\in\mathbb{R}$
it holds that
$$\phi(\varphi_r(x))(t)=\psi(\varphi_t(\varphi_r(x)))=\psi(\varphi_{t+r}(x))
=\phi(x)(t+r)=T_r(\phi(x))(t)$$
for all $t\in\mathbb{R}$,
the map $\phi$ is equivariant.
Thus,
as we stated,
it suffices to deal with the case that
$X$ is a compact invariant subset of
$C(\mathbb{R})^\mathbb{N}$.

\medskip

For each $j\in\mathbb{N}$ put $r_j=1/(j+1)$.
For any $j\in\mathbb{N}$ and $i\in\mathbb{N}$ with $i\le j$
we define a map $F_i^j:C(\mathbb{R})^\mathbb{N}\to L(\mathbb{R})$
as follows:
For every $f=\left(f_i\right)_{i\in\mathbb{N}}\in C(\mathbb{R})^\mathbb{N}$,
let
$$F_i^j(f)(t)=\int_t^{t+r_j}\,f_i(s)\;\text{d}s$$
for all $t\in\mathbb{R}$.

Since it holds that
$0\le F_i^j(f)\le r_j\le1$
and that
\begin{align*}
\left|F_i^j(f)(t)-F_i^j(f)(t')\right|
&=\left|\int_{t'+r_j}^{t+r_j}\,f_i(s)\;\text{d}s\,-\int_{t'}^t\,f_i(s)\;\text{d}s\,\right|\\
&\le\max\left\{\int_{t'+r_j}^{t+r_j}\,f_i(s)\;\text{d}s,\,\;\int_{t'}^t\,f_i(s)\;\text{d}s\right\}\\
&\le\left|t-t'\right|
\end{align*}
for all $t,t'\in\mathbb{R}$,
the image $F_i^j(f)$ is indeed in $L(\mathbb{R})$.

Now we define a map
$F:C(\mathbb{R})^\mathbb{N}\to L(\mathbb{R})^\mathbb{N}$
by
\begin{align*}
F(f)&=\left(F_i^j(f)\right)^{j\in\mathbb{N}}_{i\in\mathbb{N},i\le j}
=\left(\left(F_i^j(f)\right)_{1\le i\le j}\right)_{j\ge1}\\
&=\left(F_1^1(f),F_1^2(f),F_2^2(f),F_1^3(f),F_2^3(f),F_3^3(f),\dots,
F_1^j(f),F_2^j(f),\dots,F_j^j(f),\dots\right)
\end{align*}
for all $f\in C(\mathbb{R})^\mathbb{N}$.
We are going to show that
the map $F:X\to L(\mathbb{R})^\mathbb{N}$
is an equivariant topological embedding
of $X$ in $L(\mathbb{R})^\mathbb{N}$.

\medskip

For every $F_i^j$ and $r\in\mathbb{R}$,
we have
$$F_i^j(T_rf)(t)=\int_t^{t+r_j}\,f_i(s+r)\;\text{d}s\,
=\int_{t+r}^{t+r+r_j}\,f_i(s)\;\text{d}s\,
=F_i^j(f)(t+r)=T_rF_i^j(f)(t)$$
for all
$f=\left(f_i\right)_{i\in\mathbb{N}}\in C(\mathbb{R})^\mathbb{N}$
and $t\in\mathbb{R}$.
This shows that the map
$F:X\to L(\mathbb{R})^\mathbb{N}$ is equivariant.
It remains to check that
$F:X\to F(X)$ is a homeomorphism.

We first show that each map
$F_i^j:X\to F_i^j(X)$ is continuous.
To do this,
we take an interval $[a,b]$ and a sequence of functions
$f^{(n)}=(f_k^{(n)})_{k\in\mathbb{N}}\in C(\mathbb{R})^\mathbb{N}$
uniformly converging to
$g=(g_k)_{k\in\mathbb{N}}\in C(\mathbb{R})^\mathbb{N}$
on the interval $[a,b+1]$
as $n\to+\infty$.
This implies that the sequence $f_i^{(n)}\in C(\mathbb{R})$
tends to $g_i\in C(\mathbb{R})$ uniformly on the interval
$[a,b+1]$ as $n\to+\infty$.
Therefore the sequence
$$F_i^j(f^{(n)})(t)=\int_t^{t+r_j}\,f^{(n)}_i(s)\;\text{d}s$$
converges to
$$F_i^j(g)(t)=\int_t^{t+r_j}\,g_i(s)\;\text{d}s$$
uniformly on $t\in[a,b]$ as $n\to+\infty$.
Hence $F_i^j:X\to F_i^j(X)$ is continuous.
So $F:X\to F(X)$ is continuous.

To show that
$F:X\to F(X)$ is one-to-one,
we take
$f=\left(f_i\right)_{i\in\mathbb{N}},g=\left(g_i\right)_{i\in\mathbb{N}}
\in C(\mathbb{R})^\mathbb{N}$
and suppose that $f\ne g$.
Without loss of generality,
there exist $i\in\mathbb{N}$
and two real numbers $a<b$ such that
$f_i(s)>g_i(s)$ for all $a<s<b$.
Choose $t\in\mathbb{R}$ and a natural number $j>i$
with $a<t<t+r_j<b$.
It follows that
$$F_i^j(f)(t)=\int_t^{t+r_j}\,f_i(s)\;\text{d}s\,
>\int_t^{t+r_j}\,g_i(s)\;\text{d}s\,=F_i^j(g)(t),$$
which implies that $F(f)\ne F(g)$.
Thus,
$F:X\to F(X)$ is one-to-one.

Finally,
since $X$ is compact and
$F:X\to F(X)$ is continuous and one-to-one,
we know that the map $F:X\to F(X)$ is a homeomorphism,
which completes the proof.\end{proof}

We would like to remark that
the above constructive proof indicates that
the universal space actually can be chosen
as $\left(L(\mathbb{R})\cap C^1(\mathbb{R})\right)^\mathbb{N}$.
Moreover, replacing $C(\mathbb{R})$ by $L(\mathbb{R})\cap C^1(\mathbb{R})$
and iterating the process, we see that
$\left(L(\mathbb{R})\cap C^n(\mathbb{R})\right)^\mathbb{N}$
is universal for any finite positive integer $n$.
One may ask if $\left(L(\mathbb{R})\cap C^\infty(\mathbb{R})\right)^\mathbb{N}$
is universal.
The answer is positive.
But it requires a different construction
that we plan on publishing elsewhere.
We keep the present paper
in the very elementary mathematical analysis level.

\bigskip

\textbf{Acknowledgements.}
We thank the anonymous referee for his/her helpful
comments/suggestions.
L.J. was partially supported by NNSF of China (11371339 and 11571335).

\end{document}